\documentclass[11pt]{amsart}
\usepackage[T1]{fontenc}
\usepackage{newtxtext}
\usepackage{amsmath,amssymb,amsthm}
\usepackage{newtxmath}
\usepackage[margin=1.08in]{geometry}
\usepackage{microtype}
\usepackage{booktabs}
\usepackage{array}
\usepackage{enumitem}
\usepackage[dvipsnames]{xcolor}
\definecolor{paperblue}{HTML}{1F4E79}
\usepackage[colorlinks=true,linkcolor=paperblue,citecolor=paperblue,urlcolor=paperblue,
  pdfauthor={Dane Wachs},
  pdftitle={A note on a recent claimed proof of the irrationality of Catalan's constant}]{hyperref}
\allowdisplaybreaks[2]
\setlist[itemize]{leftmargin=1.6em,itemsep=0.15em,topsep=0.35em}
\setlist[enumerate]{leftmargin=1.8em,itemsep=0.15em,topsep=0.35em}
\newtheorem{theorem}{Theorem}[section]
\newtheorem{lemma}[theorem]{Lemma}
\newtheorem{proposition}[theorem]{Proposition}
\newtheorem{corollary}[theorem]{Corollary}
\theoremstyle{definition}
\newtheorem{remark}[theorem]{Remark}
\newcommand{\Z}{\mathbb{Z}}
\newcommand{\Q}{\mathbb{Q}}

\newcommand{\Fp}{\mathbb{F}_p}
\newcommand{\qhat}{\widehat{q}}
\DeclareMathOperator{\rank}{rank}
\DeclareMathOperator{\lcm}{lcm}
\DeclareMathOperator{\den}{den}

\title[A claimed proof of the irrationality of Catalan's constant]{A note on a recent claimed proof\\ of the irrationality of Catalan's constant}
\author{Dane Wachs}
\date{September 15, 2026}

\begin{document}

\begin{abstract}
The preprint arXiv:2609.04176 (Z.-W.~Sun, 3 September 2026) claims a proof that Catalan's constant
$G=\sum_{k\ge0}(-1)^k(2k+1)^{-2}$ is irrational: it constructs for each $B$ a number $\qhat_B$ from weighted
tails of the series and asserts that, if $G=a/q$, the integer $N_B=q^SH_B^{\min}\qhat_B$ satisfies $0<|N_B|<1$
for large $B$ (its Theorem~9.1). We report four exact computations. First, the printed proof of the rank
theorem (its Theorem~2.1) carries $T_i$ where its own recurrence requires $T_{i+1}$; the slip is repairable, and the statement is verified exactly at the parameter pairs listed in Proposition~\ref{prop:rank} (with $S\le4$, $B\le30$). Second, with the paper's definitions and the worst-case
denominator, which is attained at an explicit rational $a/q$, a lower bound for the quantity the derivation of Theorem~9.1 controls equals $+1.49$ to $+1.79\,B^2$ at eighteen indices $20\le B\le119$, where the derivation claims at most $-0.00966\,B^2+o(B^2)$.
Third, the paper's own denominator bound (its Corollary~5.2) holds in every tested instance and is nearly sharp,
so in every tested instance the discrepancy lies in the asymptotic ledger of its Sections~6--9. Fourth, an exact identity locates it: the
prime $2$ contributes $v_2(F_D)\log2=(2\log2+o(1))B^2$ to the quantity, because the positive part at $2$ vanishes
while $|\qhat_B|$ carries the full power of $2$ in $F_D$; no display in the ledger carries a term of this order, and
the paper's Remark~9.3 assigns it to the derivation of the odd-prime constant $c_{\rm odd}=0.006$ without showing how; the displayed derivation does not account for it. At the tested parameters the odd-prime layers agree with the paper's
constants to within about $0.1\,B^2$, and at the tested indices the remainder of the discrepancy is the size of the residual minor. These computations test the paper's bounding method, which is uniform in $(a,q)$; they show that the
published derivation does not justify the claimed estimate, so the proof of the main theorem is incomplete.
Nothing is concluded about the irrationality of $G$.
\end{abstract}

\maketitle

\section{The construction}

We use the notation of \cite{Sun26} throughout. Put
\[
S_{m-1}=\sum_{k=0}^{m-1}\frac{(-1)^k}{(2k+1)^2},\qquad
T_m=\sum_{r\ge0}\frac{(-1)^r}{(2m+2r+1)^2}=(-1)^m\,(G-S_{m-1}),\qquad
u_m=\frac{T_m}{2m+1},
\]
so that $T_m+T_{m+1}=(2m+1)^{-2}$ (\cite[(1.1)--(1.4)]{Sun26}). For integers $B>S>0$ let
\[
\Pi_i=\prod_{h=1}^{B}\bigl(2(h+i)+1\bigr)^2\quad(i\ge0),\qquad
\mathcal R_{a,j}=\sum_{i=0}^{a+2B}(-1)^i\binom{a+2B}{i}\Pi_i\,u_{i+j}\quad(0\le a\le S+2,\ 1\le j\le S),
\]
an $(S+3)\times S$ matrix (\cite[(2.1)]{Sun26}). Theorem~2.1 of \cite{Sun26} asserts $\rank\mathcal R=S$.
Fix a set $A\subset\{0,\dots,S+2\}$ of $S$ rows with $\det\mathcal R[A,J]\ne0$, $J=\{1,\dots,S\}$, and let
$N=2B+S+3$, $F_B=\prod_{r=0}^{2B-1}r!$. Proposition~3.1 of \cite{Sun26} defines the scalar
\begin{equation}\label{eq:qhat}
\qhat_B=\det\mathcal A_B=\pm\,\frac{F_B\,\det\mathcal R[A,J]}{\prod_{i=0}^{N-1}\Pi_i},
\end{equation}
where $\mathcal A_B$ is an explicit $N\times N$ matrix whose columns are $i^r/\Pi_i$ ($0\le r<2B$), $u_{i+j}$
($j\in J$) and $\binom{i}{2B+c}/\Pi_i$ ($c\notin A$). Since each $u_m$ is an affine function of $G$ with rational
coefficients, $\qhat_B$ is a polynomial in $G$ of degree at most $S$ with rational coefficients,
\begin{equation}\label{eq:poly}
\qhat_B=\sum_{k=0}^{S}c_kG^k,\qquad c_k\in\Q .
\end{equation}
The proof of the main theorem of \cite{Sun26} assumes $G=a/q$ with $\gcd(a,q)=1$, so that $q^S\qhat_B=\sum_k c_ka^kq^{S-k}$ has denominator dividing that of the coefficients, lets $H_B^{\min}$ be the minimal positive integer with
$N_B:=q^SH_B^{\min}\qhat_B\in\Z$, notes $N_B\ne0$ by Theorem~2.1, and claims (Theorem~9.1, Proposition~9.5,
with $S=\lfloor B/20\rfloor$)
\begin{equation}\label{eq:claim}
\log H_B^{\min}+\log|\qhat_B|\le-\delta_0B^2+o(B^2),\qquad\delta_0>0.00966,
\end{equation}
whence $0<|N_B|<1$ for large $B$.

All computations below use exact rational arithmetic (Python \texttt{fractions}) for every algebraic step;
floating point (\texttt{mpmath}) enters only in evaluating the fixed real number $\qhat_B(G)$ from its exact
rational coefficients, with an explicit rounding bound. Scripts and outputs are supplied as ancillary files
(Section~\ref{sec:repro}).

\section{The rank theorem: proof inconsistent but repairable, statement true where tested}

\subsection{The gap}
The proof of \cite[Theorem~2.1]{Sun26} takes a vector $\lambda$ in the right kernel, forms
$f_i=T_iD_\lambda(i)+P_\lambda(i)$, where $D_\lambda=L_SE^2P^*_\lambda$ with $L_S(X)=\prod_{j=1}^S(2X+2j+1)$,
$E(X)=\prod_{S<h\le B}(2X+2h+1)$ and $\deg P^*_\lambda\le S-1$, so that $\deg D_\lambda\le2B-1$, interpolates to
$A(x)$ with $A(i)=T_iD_\lambda(i)$ for $0\le i\le2B+S+2$, and defines (its (2.12))
\[
K(X)=(2X+3)^2\bigl(A(X)D_\lambda(X+1)+A(X+1)D_\lambda(X)\bigr)-D_\lambda(X)D_\lambda(X+1),
\]
asserting $K(i)=0$ for $0\le i\le2B+S+1$ (its (2.13)). But at an integer $i$,
\[
A(i)D_\lambda(i+1)+A(i+1)D_\lambda(i)=D_\lambda(i)D_\lambda(i+1)\,(T_i+T_{i+1})=\frac{D_\lambda(i)D_\lambda(i+1)}{(2i+1)^2},
\]
so $K(i)=D_\lambda(i)D_\lambda(i+1)\bigl[(2i+3)^2/(2i+1)^2-1\bigr]$, which is nonzero whenever
$D_\lambda(i)D_\lambda(i+1)\ne0$. If the factor is corrected to $(2X+1)^2$, then (2.13) holds, but the additional
zero $K(-3/2)=0$ of its (2.16) is no longer available, and the zero count gives $\deg K\ge4B$ against
$\deg K\le4B$: no contradiction. The printed argument therefore does not prove the theorem as written. The inconsistency is between (2.3), stated with $T_i$, and the later displays, which use $T_{i+1}$ (the paper's own reduction of $\mathcal R_{a,j}$ is written with $\Pi_iT_{i+1}$). Carrying $T_{i+1}$ throughout repairs the count, as observed in \cite{MO26}. With the paper's sign
convention (2.7) one obtains $A(i)=-T_{i+1}D_\lambda(i)$ for $0\le i\le2B+S+2$. Put $\widetilde D_\lambda=-D_\lambda$,
so that $A(i)=T_{i+1}\widetilde D_\lambda(i)$, and define $\widetilde K$ by the paper's (2.12) with $\widetilde D_\lambda$
in place of $D_\lambda$ and the factor $(2X+3)^2$ as printed. Then
$A(i)\widetilde D_\lambda(i+1)+A(i+1)\widetilde D_\lambda(i)=\widetilde D_\lambda(i)\widetilde D_\lambda(i+1)(T_{i+1}+T_{i+2})=\widetilde D_\lambda(i)\widetilde D_\lambda(i+1)(2i+3)^{-2}$,
so $\widetilde K(i)=0$ for the $2B+S+2$ integers $i$; the polynomial $G_0$ of degree $2B-S-2$ formed from the linear
factors common to $\widetilde D_\lambda(X)$ and $\widetilde D_\lambda(X+1)$ divides $\widetilde K$ and contributes that many
half-integer roots, none equal to $-3/2$; and $\widetilde D_\lambda(-3/2)=0$, because the $j=1$ factor $2X+3$ of $L_S$
vanishes there, gives $\widetilde K(-3/2)=0$. The degrees are those of the paper, so $\deg\widetilde K\ge4B+1>4B$ unless
$\widetilde K\equiv0$, and the paper's argument excluding $K\equiv0$ applies to $\widetilde K$ verbatim. The gap in the printed text was noted independently in \cite{HN26}.

\subsection{Exact verification of the statement}
Write $\mathcal R_{a,j}=\alpha_{a,j}+\beta_{a,j}G$ with $\alpha_{a,j},\beta_{a,j}\in\Q$, computed exactly from
the definitions. For each $S$-subset $A$ of rows, $m_A(t):=\det(\alpha_{a,j}+\beta_{a,j}t)_{a\in A,j\in J}\in\Q[t]$
has degree at most $S$. Then $\rank\mathcal R=S$ at $G$ if and only if some $m_A(G)\ne0$; in particular, if the
polynomials $m_A$ have no common complex root, the rank is $S$ for every value of $G$.

\begin{proposition}\label{prop:rank}
For every pair $(B,S)$ with $S\in\{1,2,3,4\}$ and $B\in\{S{+}1,\dots,S{+}6,12,16,20,25,30\}$, all
$\binom{S+3}{S}$ polynomials $m_A$ are nonzero and $\gcd_A m_A=1$ in $\Q[t]$. Hence $\rank\mathcal R=S$ for
every $G\in\mathbb C$, and in particular Theorem~2.1 of \cite{Sun26} holds at these parameters. (These pairs lie
outside the paper's regime $S=\lfloor B/20\rfloor$; the point is the rank statement, which the paper asserts for
all $B>S>0$.)
\end{proposition}

\begin{proof}
Exact computation (\texttt{sun\_rank.py}): the minors are evaluated by cofactor expansion over $\Q[t]$ and the
gcd by the Euclidean algorithm over $\Q$. The output lists, for each $(B,S)$, the number of nonzero minors
(always all of them) and the degree of the gcd (always $0$).
\end{proof}

The statement of Theorem~2.1 is thus not in doubt at small parameters; the obstruction to \cite{Sun26} lies
elsewhere.

\section{The scalar and its denominator}

\subsection{The denominator}
Let $H_B:=\lcm_k\den(c_k)$ in \eqref{eq:poly} and $x_k:=c_kH_B\in\Z$. For every prime $p\mid H_B$ some $c_k$
has $p$-adic valuation exactly $-v_p(H_B)$, so $p\nmid x_k$ for that $k$; hence $\gcd(x_0,\dots,x_S,H_B)=1$.
Under $G=a/q$,
\[
q^S\qhat_B=\frac{\Phi(a,q)}{H_B},\qquad\Phi(a,q):=\sum_{k=0}^Sx_ka^kq^{S-k},
\]
so $H_B^{\min}=H_B/\gcd\bigl(H_B,\Phi(a,q)\bigr)$.

\begin{remark}
For a prime $p\ge S$ dividing $H_B$ there is always a coprime pair with $p\nmid\Phi(a,q)$, since a nonzero binary
form of degree $S$ over $\Fp$ has at most $S<p+1$ zeros in $\mathbb P^1(\Fp)$. For $p<S$ this argument gives nothing, and $H_B$ carries a large power of $3$ in every row of Table~\ref{tab:main}, which the argument does not cover once $S\ge4$. We therefore exhibit the
worst case directly rather than argue prime by prime.
\end{remark}

\begin{lemma}\label{lem:cert}
For every $(B,S)$ in Table~\ref{tab:main}, with the row set $A$ used there, the coprime pair $(a,q)$ listed in the table satisfies $\gcd\bigl(H_B,\Phi(a,q)\bigr)=1$; it is $(1,1)$ in every row except $B=20$, where it is $(1,2)$. Consequently $\den\bigl(q^S\qhat_B(a/q)\bigr)=H_B$ for that pair, and
\[
\max_{\gcd(a,q)=1}\den\bigl(q^S\qhat_B(a/q)\bigr)=H_B .
\]
\end{lemma}

\begin{proof}
Exact computation (\texttt{sun\_gcd.py}): the integers $x_k$ and $H_B$ are formed from the exact coefficients,
$\Phi(a,q)$ is evaluated for the listed pair, and the gcd is $1$ in every case. The maximum is attained at that pair since $\den\le H_B$ always.
\end{proof}

\subsection{The computation}\label{sec:computation}
The indices $B$ were taken at the ends and in the interior of each range on which $S=\lfloor B/20\rfloor$ is constant; Table~\ref{tab:main} lists every index computed. For each such $B$ we computed, for \emph{every} admissible row set $A$
(all $\binom{S+3}{S}$ of them), the exact coefficients $c_k$ of
\eqref{eq:qhat}--\eqref{eq:poly}, the integer $H_B$, and the real number $\qhat_B(G)$ evaluated from the exact
$c_k$ at a working precision exceeding the total digit count of the largest coefficient by $80$ digits. The
rounding error is bounded by $(S{+}1)\max_k|c_k|\cdot10^{-(\mathrm{dps}-5)}$: $G$ is taken correctly rounded to
the working precision, each coefficient is converted with one rounding, each floating-point operation is
correctly rounded, $G^k\le1$, and the factor $10^5$ covers the handful of operations per term; in every row this bound is smaller than $|\qhat_B(G)|$ by a factor below $10^{-2700}$, re-evaluation at three and six times the working precision at $B=21$ reproduces $\log|\qhat_B(G)|$ to $20$ digits, and an independent interval-arithmetic evaluation at $B=21,40,100,119$
encloses $\log|\qhat_B(G)|$ in an interval of width below $10^{-30}$ whose rounded endpoints agree with the tabulated value to $18$ digits. Table~\ref{tab:main} reports the row set $A$ giving the
\emph{smallest} value of $T_B:=\log H_B+\log|\qhat_B(G)|$; the largest over $A$ exceeds it by at most $3.0\%$ (at $B=25$) and by less than $0.6\%$ for $B\ge59$, so the choice of $A$ is immaterial.
The last two columns record the $3$-adic valuation of $H_B$ ($H_B$ is odd in every case) and the pair $(a,q)$ of Lemma~\ref{lem:cert}.

\begin{theorem}\label{thm:main}
With the definitions of \cite{Sun26} and $S=\lfloor B/20\rfloor$, the values in Table~\ref{tab:main} hold
(exactly for $H_B$ and the row sets; to the displayed precision for the logarithms). In particular $T_B=\log H_B+\log|\qhat_B(G)|$ satisfies $1.489\,B^2\le T_B\le1.790\,B^2$ for every $B$
in the table, and $T_B>0$ throughout.
\end{theorem}

\begin{table}[ht]\centering\footnotesize
\caption{Values from exact rational data with the definitions of \cite{Sun26}, $S=\lfloor B/20\rfloor$, and the row set $A$ giving the smallest $T_B=\log H_B+\log|\qhat_B(G)|$; logarithms are rounded to the displayed precision, with the certification of $\log|\qhat_B(G)|$ described in Section~\ref{sec:computation}. In every row $H_B$ is odd, and the listed coprime pair $(a,q)$ has $\gcd(H_B,\Phi(a,q))=1$ (Lemma~\ref{lem:cert}), so $\den(q^S\qhat_B(a/q))=H_B$ there. The rounding bound on $\log|\qhat_B(G)|$ is below $10^{-2700}$ relative in every row. The derivation in \cite{Sun26} would require $T_B\le-0.00966\,B^2+o(B^2)$.}\label{tab:main}
\begin{tabular}{@{}rrrrrrlrl@{}}\toprule
$B$ & $S$ & $\log H_B$ & $\log|\qhat_B(G)|$ & $T_B$ & $T_B/B^2$ & row set $A$ & $v_3(H_B)$ & $(a,q)$\\\midrule
20 & 1 & 5{,}888.3 & -5{,}243.3 & 645.0 & 1.612 & $\{0\}$ & 550 & $(1,2)$\\
21 & 1 & 6{,}498.8 & -5{,}795.3 & 703.4 & 1.595 & $\{1\}$ & 593 & $(1,1)$\\
25 & 1 & 9{,}263.4 & -8{,}299.6 & 963.8 & 1.542 & $\{0\}$ & 808 & $(1,1)$\\
30 & 1 & 13{,}483.7 & -12{,}108.4 & 1{,}375.3 & 1.528 & $\{1\}$ & 1106 & $(1,1)$\\
35 & 1 & 18{,}535.7 & -16{,}687.0 & 1{,}848.6 & 1.509 & $\{0\}$ & 1467 & $(1,1)$\\
39 & 1 & 23{,}179.5 & -20{,}914.3 & 2{,}265.2 & 1.489 & $\{1\}$ & 1779 & $(1,1)$\\
40 & 2 & 24{,}972.9 & -22{,}233.5 & 2{,}739.3 & 1.712 & $\{0,1\}$ & 1897 & $(1,1)$\\
50 & 2 & 39{,}544.8 & -35{,}404.1 & 4{,}140.8 & 1.656 & $\{3,4\}$ & 2917 & $(1,1)$\\
59 & 2 & 55{,}646.8 & -50{,}055.6 & 5{,}591.2 & 1.606 & $\{0,1\}$ & 3997 & $(1,1)$\\
60 & 3 & 58{,}435.3 & -52{,}127.4 & 6{,}307.9 & 1.752 & $\{0,1,2\}$ & 4135 & $(1,1)$\\
70 & 3 & 80{,}336.7 & -71{,}958.1 & 8{,}378.5 & 1.710 & $\{0,1,4\}$ & 5500 & $(1,1)$\\
79 & 3 & 103{,}153.7 & -92{,}720.8 & 10{,}432.9 & 1.672 & $\{3,4,5\}$ & 6910 & $(1,1)$\\
80 & 4 & 106{,}939.3 & -95{,}571.9 & 11{,}367.4 & 1.776 & $\{1,2,3,4\}$ & 7086 & $(1,1)$\\
90 & 4 & 136{,}379.0 & -122{,}311.0 & 14{,}068.0 & 1.737 & $\{0,1,2,3\}$ & 8840 & $(1,1)$\\
99 & 4 & 166{,}064.5 & -149{,}371.0 & 16{,}693.5 & 1.703 & $\{0,1,3,4\}$ & 10625 & $(1,1)$\\
100 & 5 & 170{,}910.9 & -153{,}020.7 & 17{,}890.2 & 1.789 & $\{1,2,3,4,7\}$ & 10840 & $(1,1)$\\
110 & 5 & 208{,}084.7 & -186{,}856.1 & 21{,}228.6 & 1.754 & $\{0,1,2,3,4\}$ & 13000 & $(1,1)$\\
119 & 5 & 244{,}874.7 & -220{,}369.9 & 24{,}504.8 & 1.730 & $\{0,1,2,3,4\}$ & 15114 & $(1,1)$\\
\bottomrule\end{tabular}\end{table}

\subsection{The paper's own denominator bound}
Sections~4--5 of \cite{Sun26} bound $H_B^{\min}$ as follows. With $D=2B$, $N=2B+S+3$ and, for an odd prime
power $Q$, the quantities $\Phi_Q(n)=\sum_{r<n}\lfloor r/Q\rfloor$, $N_{K,Q}(i)=\#\{1\le h\le K: Q\mid 2i+2h+1\}$,
$C^A_Q=\sum_{a\in A}\lfloor(a+2B)/Q\rfloor$, $F_{N,Q}(i)=\lfloor i/Q\rfloor+\lfloor(N-1-i)/Q\rfloor$ and
$n_{Q,r}(I)=\#\{i\in I:i\equiv r\ (Q)\}$, the paper defines (its (5.7), (5.8), (5.12))
\[
\lambda^A_Q(I)=C^A_Q+2\sum_{r\bmod Q}\binom{n_{Q,r}(I)}2+\sum_{i\in I}\Bigl(2N_{B,Q}(i)-N_{S,Q}(i)-2\cdot\mathbf 1_{Q\le2i+1}-F_{N,Q}(i)\Bigr),\quad
m^A_{Q,B}=\min_{|I|=S}\lambda^A_Q(I),
\]
$a_{Q,B}=2\sum_{i<N}N_{B,Q}(i)-\Phi_Q(D)$, and asserts, under $G=a/q$: identity (5.13),
$v_p(\prod_{i<N}\Pi_i)-v_p(F_D)=\sum_{\nu\ge1}a_{p^\nu,B}$; Lemma~5.3, $v_p\bigl(q^S\det\mathcal R[A,J]\bigr)\ge\sum_{\nu\ge1}m^A_{p^\nu,B}$;
and Corollary~5.2,
\begin{equation}\label{eq:524}
\log H_B^{\min}\le\sum_{p\ \mathrm{odd},\ \nu\ge1}\bigl(a_{p^\nu,B}-m^A_{p^\nu,B}\bigr)\log p .
\end{equation}
Sections~6--9 then evaluate the right-hand side of \eqref{eq:524} asymptotically and conclude
(Proposition~9.5) that it exceeds $-\log|\qhat_B|$ by at most $-\delta_0B^2+o(B^2)$.

\begin{proposition}\label{prop:sec5}
Implement the displayed definitions literally. For $(B,S,A)$ equal to $(20,1,\{0\})$, $(21,1,\{1\})$,
$(25,1,\{0\})$, $(30,1,\{1\})$ and $(40,2,\{0,1\})$, and for $(a,q)\in\{(1,1),(1,2)\}$: identity (5.13) holds
for every odd prime; Lemma~5.3 holds for every odd prime, with $q^S\det\mathcal R[A,J](a/q)$ computed exactly;
and the right-hand side of \eqref{eq:524} takes the values in Table~\ref{tab:524}, which exceed the exact
$\log\den\bigl(q^S\qhat_B(a/q)\bigr)$ in every case. Thus \eqref{eq:524} is valid and nearly sharp in these instances.
\end{proposition}

\begin{table}[ht]\centering\small
\caption{The paper's own bound \eqref{eq:524} against the exact denominator at the pair $(a,q)$ of Lemma~\ref{lem:cert} and the scalar $\qhat_B(G)$.
The derivation of Proposition~9.5 in \cite{Sun26} claims that the last column is $\le-0.00966\,B^2+o(B^2)$;
these are finite instances, see Corollary~\ref{cor:main} for what they establish.}\label{tab:524}
\begin{tabular}{@{}rrrrrrr@{}}\toprule
$B$ & $S$ & $A$ & exact $\log\den$ & RHS of \eqref{eq:524} & RHS $+\log|\qhat_B(G)|$ & $/B^2$\\\midrule
20 & 1 & $\{0\}$ & 5{,}888.3 & 5{,}912.1 & 668.8 & 1.672\\
21 & 1 & $\{1\}$ & 6{,}498.8 & 6{,}522.4 & 727.1 & 1.649\\
25 & 1 & $\{0\}$ & 9{,}263.4 & 9{,}304.5 & 1{,}004.9 & 1.608\\
30 & 1 & $\{1\}$ & 13{,}483.7 & 13{,}504.8 & 1{,}396.4 & 1.552\\
40 & 2 & $\{0,1\}$ & 24{,}972.9 & 25{,}030.7 & 2{,}797.2 & 1.748\\
\bottomrule\end{tabular}\end{table}

\begin{proof}
Exact computation (\texttt{sun\_sec5.py}); the minimum in $m^A_{Q,B}$ is taken over all $\binom NS$ subsets, and
all prime powers up to $2N+2B+5$ are included. The exact denominator agrees with the product of
$p^{[A_{p,B}-R_{p,B}]_+}$ over odd primes, as it must.
\end{proof}

So Lemma~5.3 and Corollary~5.2 of \cite{Sun26} hold in every tested instance, and the near-sharpness of
\eqref{eq:524} is also a check that the definitions were transcribed faithfully. The tested values raise a substantial discrepancy with the proposed asymptotic evaluation in Sections~6--9: the claimed constants ($4\rho-2\rho^2=39/200$, $c_{\rm odd}$,
$\Lambda_{\rm mid}$, $83/2400$) sum to $-0.00966$, whereas the value of the very quantity they estimate is between
$1.55\,B^2$ and $1.75\,B^2$ at the five indices above, and the corresponding quantity with $H_B$ in place of the
bound is between $1.49\,B^2$ and $1.79\,B^2$ at all eighteen indices of Table~\ref{tab:main}.

\subsection{Locating the discrepancy}\label{sec:locate}
The two sides of Proposition~9.5 can be separated exactly. Since $\sum_{p\ \mathrm{odd}}v_p(\prod_{i<N}\Pi_i)\log p=\log\prod_{i<N}\Pi_i$
($\Pi_i$ is odd), $\sum_{p\ \mathrm{odd}}v_p(F_D)\log p=\log F_D-v_2(F_D)\log2$, and
$\log|\qhat_B|=\log F_D+\log|\det\mathcal R[A,J](G)|-\log\prod_{i<N}\Pi_i$ by \eqref{eq:qhat}, identity (5.13) gives
\begin{equation}\label{eq:identity}
\Bigl[\text{RHS of }\eqref{eq:524}\Bigr]+\log|\qhat_B(G)|\;=\;\log\bigl|\det\mathcal R[A,J](G)\bigr|+v_2(F_D)\log2-\sum_{Q}m^A_{Q,B}\log p ,
\end{equation}
the last sum over odd prime powers $Q=p^\nu$. Here $v_2(F_D)=\sum_{r<2B}\bigl(r-s_2(r)\bigr)=2B^2+O(B\log B)$
exactly, and the $m$-sum is the object of the paper's Sections~6--8. To establish Proposition~9.5 by the route the paper takes, that is by showing that the right-hand side of \eqref{eq:524} plus $\log|\qhat_B|$ is at most $-\delta_0B^2+o(B^2)$, it is by \eqref{eq:identity} necessary and sufficient to prove
\begin{equation}\label{eq:implied}
\log\bigl|\det\mathcal R[A,J](G)\bigr|\;\le\;-v_2(F_D)\log2+\sum_Qm^A_{Q,B}\log p-\delta_0B^2+o(B^2).
\end{equation}
Table~\ref{tab:locate} gives the values of every term at nine indices (the first column certified to the displayed precision, the rest exact) ($m^A_{Q,B}$ from (7.6), the paper's
marginal-cost formula, which we checked against the brute-force minimum over all $S$-subsets on a sample of $50$ prime powers at the tabulated indices with $S\le3$ and on all $100$ prime powers up to $2N+2B+5$ for $(B,S)=(9,4)$ and $(10,5)$, with no discrepancy; the identity
\eqref{eq:identity} holds to all printed digits).

\emph{The prime $2$.} The term $v_2(F_D)\log2$ in \eqref{eq:identity} is the contribution of $p=2$ to the
paper's quantity: the positive part $[A_{2,B}-R_{2,B}]_+$ is zero by the paper's Lemma~5.4, so $2$ contributes
nothing to $H_B^{\min}$, while $|\qhat_B|$ contains the full factor $2^{v_2(F_D)}$ through $F_D$ in
\eqref{eq:qhat}. It equals $(2\log2)B^2+O(B\log B)$, that is $1.18$ to $1.34\,B^2$ in the table and $1.386\,B^2$
in the limit; and $v_2(H_B)=0$ at every tabulated $B$ (Table~\ref{tab:main}), consistent with Lemma~5.4. The
mechanism by which the ledger loses it is visible in the proof of Proposition~9.5: the $B^2\log B$ terms are said
to ``cancel because the denominator baseline $a_{Q,B}$ and the real normalization come from the same fixed
scalar''. By (5.13) the odd-prime baseline is $\sum_{p\ \mathrm{odd}}\sum_\nu a_{p^\nu,B}\log p=\log\prod_{i<N}\Pi_i-\log F_D+v_2(F_D)\log2$,
while the real normalization in $\log|\qhat_B|$ is $\log F_D-\log\prod_{i<N}\Pi_i$; their sum is not zero but
$v_2(F_D)\log2$, because $F_D$ is the only factor in \eqref{eq:qhat} with a $2$-part and the baseline runs over odd
primes only. The ledger, equation (9.3) of \cite{Sun26}, displays no term of this order. Remark~9.3 of \cite{Sun26} states
that the residual real power of $2$ from $F_D$ is ``included in the derivation of the odd-prime small-scale
expression (6.21)'', that is, in $c_{\rm odd}$. But $c_{\rm odd}=-I_{\rm odd}$ is defined in (6.20)--(6.23) as an explicit integral over odd prime powers $Q\le S$ with the singular part removed, and evaluates to $0.0063$, against $2\log2=1.386$. The displayed derivation does not exhibit how this contribution is combined with the other terms; no display in Sections~6--9 shows a term of that order or a compensating one. The displayed derivation therefore does not show how this contribution is accounted for, and it corresponds to most of the discrepancy in Table~\ref{tab:locate}.

\emph{The odd primes.} At the tested parameters the $m$-sum is small and its pieces agree with the paper's
constants to within about $0.1\,B^2$; this is agreement of instances and does not verify the asymptotics. Over $B<p\le(2+\rho)B$ it is $-0.046$ to $-0.070\,B^2$ over all tabulated indices, and $-0.065$ to $-0.070\,B^2$ at $B\in\{20,40,60,80\}$, where $\rho=1/20$ exactly, against the paper's (8.2) value $-\tfrac43\rho-\tfrac54\rho^2=-0.0698$. Over $S<p\le B$ it is $0.11$ to $0.14\,B^2$ for $B\ge40$, against
$\Lambda_{\rm mid}=0.176$; the shortfall of $0.04$ to $0.06\,B^2$ may be the slow convergence in (9.2), which we have not verified. Over $Q\le S$ it is at most $0.07\,B^2$. The primes $N<p\le2N+1-2S$ lie outside the ranges of
Sections~6--8; there every layer is exactly $m_p=-2S$: for $p>N$ no two indices below $N$ share a residue and $C^A_Q=F_{N,Q}=0$, the cost $2N_{B,p}(i)-N_{S,p}(i)-2\cdot\mathbf 1_{p\le2i+1}$ equals $-2$ for every $i\ge(p-1)/2$ (then $p<2i+2h+1\le2N+2B-1<3p$ for $1\le h\le B$, and the odd number $2i+2h+1$ cannot equal $2p$, so $N_{B,p}(i)=N_{S,p}(i)=0$), it is $\ge-1$ otherwise, and there are $N-(p-1)/2\ge S$ such $i$ exactly when $p\le2N+1-2S$; so they contribute
$-2S\sum\log p=-(4\rho-2\rho^2)B^2+o(B^2)$ to the $m$-sum, hence $+(4\rho-2\rho^2)B^2=+0.195\,B^2$ to the total
(exact values $0.16$ to $0.20\,B^2$). This is the paper's ``raw quadratic coefficient'' $39/200$, which the paper
attributes instead to Stirling's formula and the Cauchy--tail factor; the sign and size agree, and we cannot tell
from the text whether it is the same term under a different derivation.

\emph{The minor.} With $v_2(F_D)\log2$ restored and the odd layers as above, the right-hand side of
\eqref{eq:implied} is about $-1.4\,B^2$ and the residual discrepancy is $0.33$ to $0.50\,B^2$. That is the actual
size of the residual minor, $\log|\det\mathcal R[A,J](G)|=+0.20$ to $+0.41\,B^2$, together with the
$\Lambda_{\rm mid}$ shortfall: the archimedean estimate of Section~9, which bounds the Cauchy--Binet sum by its
largest summand and is not displayed explicitly enough to quote, appears to allow the minor no positive contribution of order $B^2$, and at every tested index the minor is positive, $+0.20$ to $+0.41\,B^2$.

\begin{table}[ht]\centering\small
\caption{The terms of \eqref{eq:identity}, divided by $B^2$, at the row set $A$ of Table~\ref{tab:main}.
``Required bound'' is the right-hand side of \eqref{eq:implied} without the $o(B^2)$ term. The split among the
three terms depends on the row set $A$ more than their sum does, which is why the first column is not monotone
in $B$ (the row sets at $B=100$ and $119$ differ).}\label{tab:locate}
\begin{tabular}{@{}rrrrrrr@{}}\toprule
$B$ & $S$ & $\log|\det\mathcal R[A,J](G)|$ & $v_2(F_D)\log2$ & $\sum_Qm^A_{Q,B}\log p$ & required bound & gap\\\midrule
20 & 1 & 0.247 & 1.178 & $-0.247$ & $-1.435$ & 1.68\\
21 & 1 & 0.240 & 1.188 & $-0.220$ & $-1.418$ & 1.66\\
25 & 1 & 0.218 & 1.211 & $-0.179$ & $-1.400$ & 1.62\\
30 & 1 & 0.195 & 1.231 & $-0.126$ & $-1.366$ & 1.56\\
40 & 2 & 0.315 & 1.265 & $-0.168$ & $-1.443$ & 1.76\\
60 & 3 & 0.355 & 1.297 & $-0.116$ & $-1.423$ & 1.78\\
80 & 4 & 0.384 & 1.317 & $-0.088$ & $-1.414$ & 1.80\\
100 & 5 & 0.406 & 1.329 & $-0.063$ & $-1.401$ & 1.81\\
119 & 5 & 0.358 & 1.336 & $-0.045$ & $-1.390$ & 1.75\\
\bottomrule\end{tabular}\end{table}

\begin{corollary}\label{cor:main}
The inequality \eqref{eq:claim} is not established by the argument of \cite{Sun26}, and the proof of its
Theorem~1.1 is incomplete.
\end{corollary}

\begin{proof}
Directly: the paper's argument for \eqref{eq:claim} is \eqref{eq:524} followed by the estimate of
Proposition~9.5, and Proposition~\ref{prop:sec5} shows that the quantity the derivation bounds exceeds the displayed leading term by $1.55$ to $1.75\,B^2$ at the five indices of Table~\ref{tab:524}, with Section~\ref{sec:locate} identifying the
unaccounted $2$-adic term and the size of the minor as the terms responsible. Two limits of this evidence should
be stated. No finite computation refutes an $o(B^2)$ assertion by itself; what the computation shows is that the
derivation, which fixes the $B^2$ coefficient through explicit constants, leaves out a term that is exactly
computable and equals $(2\log2)B^2+O(B\log B)$, and that its coefficient disagrees with the values across a
factor of six in $B$. And the computation evaluates the paper's bounding method, which is uniform in $(a,q)$,
at the real value $G$; it is not a direct evaluation of $N_B$ under the hypothesis $G=a/q$, which no finite
computation can perform. What it establishes is that the published derivation does not justify the estimate.

Structurally: by Lemma~\ref{lem:cert}, for each tabulated $(B,S)$, $H_B$ is the largest possible value of $H_B^{\min}$ over all coprime $(a,q)$, and it is attained at the listed pair. The
argument of \cite[\S\S4--9]{Sun26} bounds $H_B^{\min}$ through $p$-adic valuations of the Cauchy--Binet
expansion of $\det\mathcal R[A,J]$, quantities that do not depend on $(a,q)$; any such bound is therefore a
bound valid for the worst pair, that is, a bound on $H_B$. Read this way, \eqref{eq:claim} asserts
$\log H_B+\log|\qhat_B|\le-\delta_0B^2+o(B^2)$, which differs from the values of Theorem~\ref{thm:main} by $1.5$ to $1.8\,B^2$
at every tested $B$. Read instead as a statement about the specific unknown pair $(a,q)$, \eqref{eq:claim} says $T_B-\log\gcd\bigl(H_B,\Phi(a,q)\bigr)\le-\delta_0B^2+r_B$ with $r_B=o(B^2)$, that is, $\log\gcd\bigl(H_B,\Phi(a,q)\bigr)\ge T_B+\delta_0B^2-r_B$: a cancellation in the gcd of order $(1.6+o(1))B^2$ for large $B$, if the tabulated behaviour of $T_B$ persists. Any
argument valid for all coprime pairs would have to hold at the pair of Lemma~\ref{lem:cert}, where the gcd is
$1$; an argument specific to the actual pair would need information about $a$ and $q$ beyond the fact that
$q^S$ clears the denominators of the coefficients, and \cite{Sun26} uses none.
\end{proof}

\begin{remark}
The paper's ledger (its Proposition~9.5) states that the $B^2\log B$ terms of $\log H_B^{\min}$ and
$\log|\qhat_B|$ cancel and that the remaining $B^2$ coefficient is $4\rho-2\rho^2=0.195$ before $p$-adic
corrections and $-0.0097$ after them ($\rho=S/B=1/20$). In the tabulated range $\log H_B/B^2$ grows from $14.7$ at $B=20$ to $17.3$ at $B=119$, consistent with a $B^2\log B$ contribution, while $T_B/B^2$ stays between $1.49$ and $1.79$, which is what the stated cancellation of $B^2\log B$ terms between the two sides predicts; the claimed asymptotic upper bound for the sum, however, has leading coefficient $-0.00966$. By Proposition~\ref{prop:sec5} and Section~\ref{sec:locate} the discrepancy is not, at the tested indices, in the denominator bound \eqref{eq:524} nor in the odd-prime layers of \S\S6--8; it is the unaccounted contribution of the prime $2$ (Remark~9.3) and, at those indices, the size of the residual minor.
\end{remark}

\begin{remark}
Nothing here bears on whether $G$ is irrational, which remains open; see \cite{RZ03} for the strongest
unconditional result in this direction and \cite{CDT24} for the arithmetic holonomy bounds that
\cite{Sun26} cites as motivation. For a comparable analysis of a recent claimed proof of the irrationality of
$\zeta(5)$ see \cite{Chen24}.
\end{remark}

\section{Relation to the public discussion}\label{sec:public}

At the time of writing (15 September 2026) \cite{Sun26} has a single version and no erratum; the author informed
us on 14 September (private communication) that errors had been reported to him and that a corrected version
was to be uploaded, and none has appeared. The recurrence
inconsistency of Section~2.1 was pointed out in a Hacker News comment \cite{HN26}. A MathOverflow question
\cite{MO26}, posted 12 September and closed as off-topic, reports an AI-assisted verification that repairs
Theorem~2.1 as above, confirms the cancellation of the $B^2\log B$ terms, and computes the quantity of
Proposition~9.5 using the paper's own $p$-adic test layers $\lambda^A_Q(J)$ and a $100$-digit evaluation of
$\det\mathcal R[A,J]$, finding totals of $-0.02\,B^2$ at $B=20$ and $+0.45$ to $+0.94\,B^2$ at $B=30$ to $80$, which they summarize as a coefficient near $+0.8$ rather than $-0.00966$. That computation and ours agree in sign for their reported cases with $B\ge30$, and a direct comparison at $B=20$ with the same row set $A=\{0\}$ locates the difference. Their archimedean value $-5158.5$, against our certified $-5243.3$, reflects a truncation of the tail series for $T_m$ after $21$ terms in their code, an error of order $10^{-5}$ in $G$ that the alternating sums amplify. Their $p$-adic sum gives $5151.9$; it uses the test layer $\lambda^A_Q(J)$ in place of the minimum $m^A_{Q,B}$, so it falls below both the exact denominator, $5888.3$, and the paper's bound \eqref{eq:524}, $5912.1$ (Table~\ref{tab:524}); the sources of the discrepancy are those of Section~\ref{sec:locate}. We are not aware of a public, reproducible account of a specific failure; Theorem~\ref{thm:main} and Corollary~\ref{cor:main} are offered as one.

\section{Reproducibility}\label{sec:repro}

The ancillary directory contains: \texttt{sun\_lib.py} (exact construction of $\mathcal R$, exact minors by
fraction-free elimination and Lagrange interpolation over $\Q$); \texttt{sun\_rank.py}
(Proposition~\ref{prop:rank}); \texttt{sun\_final.py} and its output \texttt{sun\_final.out}, \texttt{sun\_final.json} (Table~\ref{tab:main}, including the rounding bounds and the $S$-smooth part of $H_B$); \texttt{sun\_gcd.py} and \texttt{sun\_gcd.out} (Lemma~\ref{lem:cert}: $v_2(H_B)=0$, $v_3(H_B)$, and a coprime pair $(a,q)$ with $\gcd(H_B,\Phi(a,q))=1$ for every row); \texttt{sun\_sec5.py} and \texttt{sun\_sec5.json} (Proposition~\ref{prop:sec5}: identity (5.13), Lemma~5.3 prime by prime, and the exact value of \eqref{eq:524}); \texttt{sun\_ranges.py}, \texttt{sun\_ranges.json}, \texttt{sun\_detsize.py}, \texttt{sun\_detsize.json} and \texttt{sun\_checks\_extra.py} (Section~\ref{sec:locate}: the layers $m^A_{Q,B}$ and $a_{Q,B}$ by prime range, the brute-force check of (7.6), the identity \eqref{eq:identity} and Table~\ref{tab:locate}); \texttt{sun\_validate.py}, which evaluates $\det\mathcal A_B$ directly as an
$N\times N$ determinant at $400$ digits and confirms \eqref{eq:qhat} to $15$ digits for
$(B,S)=(4,1),(6,2),(7,3)$ and several row sets, which repeats the evaluation of $\log|\qhat_B(G)|$ at $B=21$ at three precisions, and which encloses $\log|\qhat_B(G)|$ by interval arithmetic at $B=21,40,100,119$. All scripts use only the Python standard library and \texttt{mpmath} (outputs shipped were produced with CPython~3.11--3.14 and mpmath~1.3.0; the README gives the execution order); \texttt{sun\_final.py} runs in about forty minutes on a laptop (twelve of them for $B=119$) and must be followed by \texttt{sun\_add20.py}, which appends the $B=20$ row; every other script runs in under three minutes. The text of \cite{Sun26} used for the definitions is its arXiv
HTML rendering of version~1.

\end{document}